\documentclass[12pt, a4paper, reqno]{amsart}

\usepackage{amsmath, amssymb, amsthm, amsfonts}

\usepackage[mathscr]{euscript} 
\usepackage{stmaryrd} % St Mary's Road symbols font --- some extra symbols

\usepackage[xcharter]{newtxmath}
\usepackage{graphics, stackrel}
\usepackage{graphicx}

\usepackage{verbatim}
\usepackage[numbers]{natbib}
\usepackage{enumitem}

\usepackage[linesnumbered,ruled]{algorithm2e}

\usepackage{caption}
\usepackage{subcaption}

\usepackage{booktabs}

\usepackage{fancyvrb}
\usepackage[dvipsnames,svgnames,table]{xcolor}
\usepackage{mdwlist}

\usepackage[breaklinks=true, citecolor=Blue, colorlinks=true, linkcolor=blue]{hyperref}

\usepackage{enumitem}
\setlist[enumerate]{itemsep=2pt,topsep=3pt}
\setlist[itemize]{itemsep=2pt,topsep=3pt}
\setlist[enumerate,1]{label=(\alph*)}

\usepackage{mathrsfs}  % caligraphic
\usepackage{bbm}
\usepackage{bm}        % bold symbols

\usepackage[left=1.25in, right=1.25in, top=1.0in, bottom=1.15in, includehead, includefoot]{geometry}

\renewcommand{\leq}{\leqslant}
\renewcommand{\geq}{\geqslant}

\usepackage{enumitem}
\setlist[enumerate]{itemsep=2pt,topsep=3pt}
\setlist[itemize]{itemsep=2pt,topsep=3pt}
\setlist[enumerate,1]{label={\upshape (\roman*)}}

\newcommand{\setntn}[2]{ \{ #1 : #2 \} }
\newcommand{\fore}{\therefore \quad}
\newcommand{\preqsd}{\preceq_{s} }

\newcommand{\eqdist}{\stackrel {\scriptsize{d}} {=} }
\newcommand{\iidsim}{\stackrel {\textrm{ {\sc iid }}} {\sim} }
\newcommand{\1}{\mathbbm 1}

\newcommand{\given}{\, | \,}

\newcommand*\diff{\mathop{}\!\mathrm{d}}

\let\emptyset\varnothing

\newcommand{\cC}{\mathscr C}

\newcommand{\bB}{\mathcal B}

\newcommand{\fF}{\mathscr F}

\newcommand{\RR}{\mathbbm R}

\newcommand{\QQ}{\mathbbm Q}
\newcommand{\NN}{\mathbbm N}
\newcommand{\GG}{\mathbbm G}

\newcommand{\PP}{\mathbbm P}

\newcommand{\EE}{\mathbbm E}
\newcommand{\XX}{\mathbbm X}
\newcommand{\WW}{\mathbb W}

\renewcommand{\phi}{\varphi}
\renewcommand{\epsilon}{\varepsilon}

\theoremstyle{plain}
\newtheorem{theorem}{Theorem}[section]

\newtheorem{lemma}[theorem]{Lemma}

\theoremstyle{definition}

\newcommand{\navy}[1]{\textcolor{Brown}{\emph{#1}}}

\begin{document}

\title[Stochastically Monotone Markov Chains]{Quantitative Convergence Rates for Stochastically Monotone Markov Chains}

\author{Takashi Kamihigashi}
\author{John Stachurski}

\date{\today}

%\thanks{}

\begin{abstract}
For Markov chains and Markov processes exhibiting a form of stochastic monotonicity (larger states shift up transition probabilities in terms of stochastic dominance), stability and ergodicity results can be obtained using order-theoretic mixing conditions.  We complement these results by providing quantitative bounds on deviations between distributions.  We also show that well-known total variation bounds can be recovered as a special case.
\end{abstract}

\maketitle

\section{Introduction}

Quantitative bounds on the distance between distributions generated by Markov
chains have many applications in statistics and the natural and social sciences
(see, e.g., \cite{rosenthal2023markov, montenegro2006mathematical}). One
approach uses total variation distance and exploits minorization conditions
(see, e.g., \cite{rosenthal2002quantitative, jiang2021coupling,
bardet2013total}).  Another branch of the literature bounds deviations using
Wasserstein distance \citep{chafai2010long, qin2022geometric, qu2023computable}.
These bounds require some form of uniform continuity with respect to a metric on
the state space.

In some applications, Markov chains lack both the minorization and continuity
properties discussed above, making total variation and Wasserstein-type bounds
difficult or impossible to apply. Fortunately, some of these models also possess
valuable structure in the form of stochastic monotonicity. Such monotonicity can
be exploited to obtain stability and ergodicity via order-theoretic versions of
mixing conditions \citep{bhattacharya1988asymptotics, bhattacharya1999theorem,
    foss2004overview, bhattacharya2010limit, kamihigashi2014stochastic,
foss2018stochastic, kamihigashi2019unified}. In this paper we complement these
stability and ergodicity results by providing a theorem on quantitative bounds
for stochastically monotone Markov chains.  

There already exist several results that use stochastic monotonicity to 
bound the distributions generated by Markov chains
\cite{lund1996computable, gaudio2021exponential}. However, these bounds are
typically stated in terms of total variation distance, which again requires
traditional minorization conditions (as opposed to the order-theoretic mixing
conditions discussed in the last paragraph).  In this paper, we aim to fully
exploit the monotonicity by instead bounding
Kolmogorov distance between distributions.  This works well because Kolmogorov
distance respects order structure on the state space.

Our main theorem is closely related to the total variation bound in Theorem~1 of
\cite{rosenthal2002quantitative}, which is representative of existing work on
total variation bounds and supplies a simple and elegant proof.  The main
differences between that theorem and the one presented below is that we use 
Kolmogorov distance instead of total variation distance and an
order-theoretic mixing condition instead of a standard minorization condition.
At the same time, it is possible to recover Theorem~1 of
\cite{rosenthal2002quantitative} from the result we present below by a
particular choice of partial order (see Sections~\ref{ss:tv} and \ref{ss:wmf}).
Thus our work can be viewed as a generalization of existing total variation results.

\section{Set Up}

In this section we recall basic definitions and state some preliminary results.

\subsection{Environment}

Throughout this paper, $\XX$ is a Polish space,
 $\bB$ is its Borel sets, and $\preceq$ is a closed partial order on
$\XX$.  The last statement means that the graph of $\preceq$, denoted by
\begin{equation*}
    \GG \coloneq \setntn{(x', x) \in \XX \times \XX}{x' \preceq x},
\end{equation*}
is closed under the product topology on $\XX \times \XX$.
A map $h \colon \XX \to \RR$ is called \navy{increasing} if $x \preceq x'$ implies $h(x) \leq
h(x')$. 
We take $p\bB$ to be the
set of all probability measures on $\bB$ and let $b\bB$ be the bounded Borel
measurable functions sending $\XX$ into $\RR$. 
The symbol $ib\bB$ represents all increasing $h \in b\bB$. 

Given $\mu, \nu$ in $p\bB$, we say that $\mu$ is \navy{stochastically dominated}
by $\nu$ and write $\mu \preqsd \nu$ if $\mu(h) \leq \nu(h)$ for all $h \in
ib\bB$. In addition, we set
\begin{equation}\label{eq:defk}
    \kappa(\mu, \nu) \coloneq \sup 
    \left\{
        \left| \, \int h d \mu - \int h d \nu  \,\right|
        \, : \,
        h \in ib\bB \text{ and } 0 \leq h \leq 1
    \right\},
\end{equation}
which corresponds to the Kolmogorov metric on $p\bB$
\cite{kamihigashi2019unified, gaunt2023bounding}.

A function $Q \colon (\XX, \bB) \to
\RR$ is called a \navy{stochastic kernel} on $(\XX, \bB)$ if $Q$ is a map from
$\XX \times \bB$ to $[0,1]$ such that that $x \mapsto Q(x, A)$ is measurable for
each $A \in \bB$ and $A \mapsto Q(x, A)$ is a probability measure on $\bB$ for
each $x \in \XX$.  At times we use the symbol $Q_x$ to represent the
distribution $Q(x, \cdot)$ at given $x$.  A stochastic kernel $Q$ on $(\XX,
\bB)$ is called \navy{increasing} if $Qh \in ib\bB$ whenever $h \in ib\bB$.

For a given stochastic kernel $Q$ on $(\XX, \bB)$, we define the \navy{left and
right Markov operators} generated by $Q$ via 
\begin{equation*}
    \mu Q(A) \coloneq \int Q(x, A) \mu (\diff x)
    \quad \text{and} \quad
    Qf(A) \coloneq \int f(y) Q(x, \diff y).
\end{equation*}
(The left Markov operator $\mu \mapsto \mu Q$ maps $p\bB$ to itself, while the
right Markov operator $f \mapsto Qf$ acts on bounded measurable functions.) A
discrete-time $\XX$-valued stochastic process $(X_t)_{t \geq 0}$ on a filtered
probability space $(\Omega, \fF, \PP, (\fF_t)_{t \geq 0})$ is called \navy{Markov-$(Q,
\mu)$} if $X_0 \eqdist \mu$ and
\begin{equation*}
    \EE [ h(X_{t+1}) \,|\, \fF_t ] = Qh(X_t)
    \text{ with probability one for all $t \geq 0$ and $h \in b\bB$}.
\end{equation*}

\subsection{Couplings}

A \navy{coupling} of $(\mu, \nu) \in p\bB \times p\bB$ is a probability measure
$\rho$ on $\bB \otimes \bB$ satisfying
$\rho(A \times \XX) = \mu(A) $ and $\rho(\XX \times A) = \nu(A)$
for all $A \in \bB$.  
Let $\cC(\mu, \nu)$ denote the set of all couplings of
$(\mu, \nu)$  and let
\begin{equation}\label{eq:alo}
    \alpha(\mu, \nu)
    \coloneq \sup_{\rho \in \cC(\mu, \nu)} \rho(\GG)
    \qquad ((\mu, \nu) \in p\bB \times p\bB).
\end{equation}
The value $\alpha(\mu, \nu)$ lies in $[0, 1]$ and can be understood as a
measure of ``partial stochastic dominance'' of $\nu$ over $\mu$
\cite{kamihigashi2020partial}.   By the Polish assumption and Strassen's theorem 
\citep{strassen1965existence, lindvall2002lectures} we have
\begin{equation}\label{eq:st}
    \alpha(\mu, \nu) = 1 
    \quad \text{whenever} \quad
    \mu \preqsd \nu.
\end{equation}

Let $Q$ be a stochastic kernel on $(\XX, \bB)$ and let 
$\hat Q$ be a stochastic kernel on $(\XX \times \XX, \bB \otimes \bB)$.  We call
$\hat Q$ a \navy{Markov coupling} of $Q$ if $\hat Q_{(x, x')}$ is a coupling of $Q_x$ and
$Q_{x'}$ for all $x, x' \in \XX$.  We call
$\hat Q$ a \navy{$\preceq$-maximal Markov coupling} of $Q$ if $\hat Q$
is a Markov coupling of $Q$ and, in addition,
\begin{equation}\label{eq:hq}
    \hat Q((x, x'), \GG) = \alpha(Q_x, Q_{x'})
    \quad \text{for all } (x,x') \in \XX \times \XX.
\end{equation}

\begin{lemma}\label{l:pe}
    For any stochastic kernel $Q$ on
    $(\XX, \bB)$, there exists a $\preceq$-maximal Markov coupling of $Q$.
\end{lemma}

\begin{proof}
    By Theorem~1.1 of
    \cite{zhang2000existence}, given 
    lower semicontinuous $\phi \colon \XX \times \XX \to \RR$,
    there exists a stochastic kernel $\hat Q$ on $(\XX \times \XX, \bB \otimes \bB)$ such that
    $\hat Q$ is a Markov coupling of $Q$ and, in addition
    \begin{equation*}
        (\hat Q \phi)(x, x')
        = \inf \left\{
            \int \phi \diff \rho 
            \, : \,
            \rho \in \cC(Q_x, Q_{x'})
        \right\}.
    \end{equation*}
    As $\GG$ is closed, this equality is attained when $\phi = 1 - \1_\GG$.  
    Since $\hat Q(_{(x, x')}$ and $\rho$ are probability measures,
    we then have
    \begin{equation*}
        \hat Q((x, x'), \GG) 
        = \sup
        \left\{
            \rho(\GG)
            \, : \,
            \rho \in \cC(Q_x, Q_{x'})
        \right\}.
    \end{equation*}
    Thus, $\hat Q$ is a $\preceq$-maximal Markov coupling of $Q$.
\end{proof}

\subsection{Drift}

Consider the geometric drift condition
\begin{equation}\label{eq:dc}
    QV(x) \leq \lambda V(x) + \beta
      \quad \text{for all } x \in \XX,
\end{equation}
where $Q$ is a stochastic kernel on $(\XX, \bB)$, $V$ is a measurable function
from $\XX$ to $[1,\infty)$, and $\lambda$ and $\beta$ are nonnegative
constants. We fix $d \geq 1$ and set
\begin{equation}\label{eq:gc}
    \gamma \coloneq \lambda + \frac{2 \beta}{d}
    \quad \text{and} \quad
    C \coloneq \setntn{x \in \XX}{V(x) \leq d}.
\end{equation}
Fix $\mu, \mu'$ in $p\bB$ and set
\begin{equation}\label{eq:h}
    H(\mu, \mu') \coloneq
    \frac{1}{2}
      \left[
          \int V \diff \mu + \int V \diff\mu'
      \right] .
\end{equation}
Let $\hat Q$ be a Markov coupling of $Q$ and let $((X_t, X'_t))_{t \geq 0}$ be
Markov-$(\hat Q, \mu \times \mu')$ on $(\Omega, \fF, \PP, (\fF_t)_{t \geq 0})$.
We are interested in studying the number of visits to $C \times C$, as given by
\begin{equation*}
    N_t \coloneq \sum_{j=0}^t \1 \{ (X_t, X'_t) \in C \times C \}. 
\end{equation*}

\begin{lemma}\label{l:ggc}
    If $Q$ satisfies the geometric drift condition \eqref{eq:dc}, then,
    for all $t \in \NN$ and all $j \in \NN$ with $j \leq t$, we have
    \begin{equation*}
        \PP\{ N_t < j \} \leq \gamma^t d^{j-1} H(\mu, \mu').
    \end{equation*}
\end{lemma}

The result in Lemma~\ref{l:ggc} has already been used in other sources.
To make the paper more self-contained, we provide a proof in the appendix.  Our
proof is based on arguments in
\cite{rosenthal2002quantitative}.

\section{Convergence Rates}

Let $V$ be a measurable function from $\XX$ to $[1,\infty)$  and let $Q$ be a
stochastic kernel on $(\XX, \bB)$ satisfying the geometric drift condition
\eqref{eq:dc}. Fix $d \in \RR_+$ and let $C$ and $\gamma$ be as
defined in \eqref{eq:gc}. Let $H(\mu, \mu')$ be as given in \eqref{eq:h}. Let 
\begin{equation}\label{eq:defc}
    \epsilon \coloneq \inf
        \left\{
            \alpha(Q_x, Q_{x'})
         \, : \,
         (x, x') \in C \times C
        \right\}.
\end{equation}
We now state the main result.

\begin{theorem}\label{t:bk}
    If $Q$ is increasing, then, for any $j, t \in \NN$
    with $j \leq t$,  we have
    \begin{equation*}
        \kappa( \mu Q^t, \mu' Q^t)
          \leq 
              (1-\epsilon)^j + 
              \gamma^t d^{j-1} H(\mu, \mu').
    \end{equation*}
\end{theorem}

\begin{proof}
    Given $Q$ in Theorem~\ref{t:bk}, we let
    $\hat Q$ be a $\preceq$-maximal Markov coupling of $Q$ (existence of which
    follows from Lemma~\ref{l:pe}).  Let $((X_t, X'_t))_{t \geq 0}$ be Markov-$(\hat Q, \mu \times \mu')$ on
    $(\Omega, \fF, \PP, (\fF_t)_{t \geq 0})$.
    We observe that the graph $\GG$ of $\preceq$ is
    absorbing for $\hat Q$.  Indeed, if $(x, x') \in \GG$, then, since $Q$ is increasing, $Q(x, \cdot)
    \preqsd Q(x', \cdot)$. Hence, by \eqref{eq:st}, we have $\alpha(Q(x, \cdot), Q(x', \cdot)) = 1$.  
    Applying \eqref{eq:hq} yields $\hat Q((x,x'), \GG) = 1$. 

    Let $\tau$ be the stopping time $\tau \coloneq \inf \setntn{t \geq 0}{X'_t
    \preceq X_t}$  with $\inf \emptyset = \infty$.  Since $\GG$ is absorbing for
    $\hat Q$, we have $\PP\{X_t' \preceq X_t\} = 1$ whenever $t \geq \tau$. Let
    $h$ be any element of $ib\bB$ with $0 \leq h \leq 1$. Since $((X_t,
    X'_t))_{t \geq 0}$ is Markov-$(\hat Q, \mu \times \mu')$ and $\hat Q((x,x'),
    \cdot)$ is a coupling of $Q(x, \cdot)$ and $Q(x', \cdot)$, we have
    \begin{align*}
        (\mu' Q^t)(h) - (\mu Q^t)(h)
         & = \EE h(X'_t) - \EE h(X_t) \\
         & = \EE [h(X'_t) - h(X_t)] \1\{X'_t \preceq X_t\}
          +  \EE [h(X'_t) - h(X_t)] \1\{X'_t \preceq X_t\}^c .
    \end{align*}
    Since $h$ is increasing, this leads to
    \begin{equation*}
        (\mu' Q^t)(h) - (\mu Q^t)(h) 
        \leq  \EE [h(X'_t) - h(X_t)] \1\{X'_t \preceq X_t\}^c
        \leq \PP \{X'_t \preceq X_t\}^c.
    \end{equation*}
    Since $\tau \leq t$ implies $X'_t \preceq X_t$ we have $\{X'_t \preceq
    X_t\}^c \subset \{\tau > t\}$, and hence
    \begin{equation}
        \label{eq:coupb}
        (\mu' Q^t)(h) - (\mu Q^t)(h) 
        \leq \PP \{\tau > t\}.
    \end{equation}
    Now define $N_t \coloneq \sum_{j=0}^t \1 \{ (X'_t, X_t) \in C \times C \}$.
    Fixing $j \in \NN$ with $j \leq t$, we have
    \begin{equation}
        \label{eq:bd}
        \PP\{ \tau > t \} 
          = \PP\{ \tau > t, N_t < j \} + \PP\{ \tau > t, N_t \geq j \}.
    \end{equation}
    To bound the first term in \eqref{eq:bd}, we set 
    $W(x,x') \coloneq [V(x) + V(x')]/2$. Since $\hat Q((x,x'), \cdot)$ is a
    coupling of $Q(x, \cdot)$ and $Q(x', \cdot)$, we have
    \begin{equation}\label{eq:bqv1}
          \hat Q W(x, x') 
          = \frac{Q V(x) + QV(x')}{2}
          \leq  \lambda W(x, x') + \beta.
    \end{equation}
    Hence, applying Lemma~\ref{l:ggc} to $\hat Q$ yields
    \begin{equation}\label{eq:stb}
       \PP\{ \tau > t, N_t < j \} 
       \leq \PP\{ N_t < j \}
       \leq = \gamma^t d^{j-1} H(\mu, \mu')
    \end{equation}
    Regarding the second term in \eqref{eq:bd}, we claim that
    \begin{equation}
        \label{eq:fb}
        \PP\{ \tau > t,\, N_t \geq j \} \leq (1 - \epsilon)^j.
    \end{equation}
    To see this, suppose $(J_i)_{i \geq 1}$ is the times of the
    successive visits of $(X_t,X_t')$ to $C \times C$.  That is,
        $J_1$ is the time of the first visit and  
    \begin{equation*}
        J_{i+1} \coloneq \inf \setntn{m \geq J_i + 1}{(X_m,X_m')  \in C \times C}
        \quad \text{for } i \geq 1.
    \end{equation*}
    It is not difficult to see that $\{N_t > j \} \subset \{J_j 
    \leq t - 1\}$.  As a result,
    \begin{equation}
        \label{eq:rul}
        \PP \{ \tau > t, \, N_t > j \}
        \leq \PP \{ \tau > t, \, J_j + 1 \leq t\}.
    \end{equation}
    Consider the set $\{ \tau > t, \, J_j + 1 \leq t\}$.  If a path is in
    this set, then as $\tau > t$, for any index $j$ with $j \leq t$ we have 
    $X_j' \npreceq X_j$.  In addition, $J_i + 1 \leq J_j + 1 \leq t$ for any $i
    \leq j$, so $X'_{J_i + 1} \npreceq X_{J_i + 1}$ for every $i \leq j$.
    \begin{equation}
        \label{eq:srul}
        \fore
        \PP \{ \tau > t, \, J_j + 1 \leq t\}
        \leq \PP 
          \cap_{i=1}^j \{X_{J_i + 1}' \npreceq X_{J_i + 1} \}.
    \end{equation}
    Observe that
    \begin{equation*}
        \PP
          \cap_{i=1}^j \{X_{J_i + 1}' \npreceq X_{J_i + 1} \}
        = \PP \left[
        \cap_{i=1}^{j-1} \{X_{J_i + 1}' \npreceq X_{J_i + 1} \} \;
        \PP [ X_{J_j + 1}' \npreceq X_{J_j + 1}  \, | \, \fF_{J_j} ]
          \right].
     \end{equation*}
     By the definition of $J_j$ we have  $(X_{J_j}, X'_{J_j})
     \in C \times C$.  Using this fact, 
     the strong Markov property and the definition of $\hat Q$ (see \eqref{eq:hq}) yields
     \begin{equation*}
        \PP [ X_{J_j + 1}' \preceq X_{J_j + 1}  \, | \, \fF_{J_j} ]
        = 
        \hat Q((X_{J_j}', X_{J_j}), \GG)
        = \alpha(Q(X_{J_j}', \cdot), Q(X_{J_j}, \cdot)).
     \end{equation*}
     Applying the definition of $\epsilon$ in \eqref{eq:defc}, we obtain
     $\PP [ X_{J_j + 1} \npreceq X_{J_j + 1}'  \, | \, \fF_{J_j} ] \leq 1 -
     \epsilon$, so
     \begin{equation*}
         \PP 
          \cap_{i=1}^j \{X_{J_i + 1} \npreceq X_{J_i + 1}' \}
          \leq 
            (1 - \epsilon)
        \PP \
        \cap_{i=1}^{j-1} \{X_{J_i + 1} \npreceq X_{J_i + 1}' \}.
     \end{equation*}
     Continuing to iterate backwards in this way yields
         $\PP 
          \cap_{i=1}^j \{X_{J_i + 1} \npreceq X_{J_i + 1}' \}
          \leq 
            (1 - \epsilon)^j$.
     Combining this inequality with (\ref{eq:rul}) and (\ref{eq:srul})
     verifies (\ref{eq:fb}).

    Combining \eqref{eq:coupb}, \eqref{eq:bd}, \eqref{eq:stb}, and \eqref{eq:fb} yields
    \begin{equation*}
        (\mu' Q^t)(h) - (\mu Q^t)(h) 
           \leq (1 - \epsilon)^j + 
           \gamma^t d^{j-1} \, H(\mu, \mu').
    \end{equation*}
    Reversing the roles of $\mu$ and $\mu'$ does not change the value on the
    right-hand side of this bound, and hence
    \begin{equation*}
        | (\mu' Q^t)(h) - (\mu Q^t)(h) |
           \leq (1 - \epsilon)^j + 
           \gamma^t d^{j-1} \, H(\mu, \mu')
    \end{equation*}
    also holds.  Taking the supremum over all $h \in ib\bB$ with $0 \leq h \leq
    1$ completes the proof
    of Theorem~\ref{t:bk}.
\end{proof}

\section{Examples and Applications}\label{s:sc}

In this section we consider some special cases, with a focus on (a) connections
to the existing literature and (b) how to obtain an estimate of the value
$\epsilon$ in \eqref{eq:defc}.

\subsection{Connection to Total Variation Results}\label{ss:tv}

One special case is when $\preceq$ is the identity order, so
that $x \preceq y$ if and only if $x = y$.  For this order we have $ib\bB =
b\bB$, so every stochastic kernel is increasing, and the Kolmogorov metric (see
\eqref{eq:defk}) becomes the total variation distance. In this setting total
variation setting, Theorem~\ref{t:bk} is similar to standard geometric bounds 
for total variation distance, such as Theorem~1 in \cite{rosenthal2002quantitative}.  

It is worth noting that, in the total variation setting, $\epsilon$ in
\eqref{eq:defc} is at least as large as the analogous term $\epsilon$ in
Theorem~1 in \cite{rosenthal2002quantitative}. Indeed, in
\cite{rosenthal2002quantitative}, the value $\epsilon$, which we now write as
$\hat \epsilon$ to avoid confusion, comes from an assumed minorization
condition: there exists a $\nu \in p\bB$ such that
\begin{equation}\label{eq:mc}
    \hat \epsilon \nu(B) \leq Q(x, B) 
    \quad \text{for all } B \in \bB \text{ and } x \in C.
\end{equation}
To compare $\hat \epsilon$ with $\epsilon$ defined in \eqref{eq:defc}, suppose
that this minorization condition holds and define the residual kernel
$R(x, B) \coloneq (Q(x, B) - \hat \epsilon \nu(B))/(1-\hat \epsilon)$.
Fixing $(x, x') \in C \times C$, we draw $(X, X')$ as follows:
With probability $\epsilon$, draw $X \sim \nu$ and set $X' = X$.
With probability $1-\epsilon$, independently draw $X \sim R(x, \cdot)$ and $X' \sim R(x', \cdot)$.
Simple arguments confirm that $X$ is a draw from $Q(x, \cdot)$ and $X'$ is a
draw from $Q(x', \cdot)$.  Recalling that $\preceq$ is the identity order, this
leads to 
$\hat \epsilon 
    = \PP\{I \leq \hat \epsilon\}
    \leq \PP\{X = X'\} = \PP\{X \preceq X'\} \leq \alpha(Q(x, \cdot), Q(x',
    \cdot))$.
(The last bound is by the definition of $\alpha$ in \eqref{eq:alo} and the fact
that the joint distribution of $(X, X')$ is a
coupling of $Q(x, \cdot)$ and $Q(x', \cdot)$.)
Since, in this discussion, the point $(x,x')$ was arbitrarily chosen from $C
\times C$, we conclude that $\hat \epsilon \leq \epsilon$, where $\epsilon$ is
as defined in \eqref{eq:defc}.

\subsection{Stochastic Recursive Sequences}\label{ss:srs}

The preceding section showed that Theorem~\ref{t:bk} reduces to existing results
for bounds on total variation distance when the partial order $\preceq$ is the
identity order.  Now we show how Theorem~\ref{t:bk} leads to new results other
settings, such as when $\preceq$ is a pointwise partial order.
To this end, consider the process
\begin{equation}\label{eq:srs}
    X_{t+1} = F(X_t, W_{t+1})
\end{equation}
where $(W_t)_{t \geq 1}$ is an {\sc iid} shock process taking values in some
space $\WW$, and $F$ is a measurable function from $\XX \times \WW$ to $\XX$.  The
common distribution of each $W_t$ is denoted by $\phi$.   We suppose that $F$ is
increasing, in the sense that
    $x \preceq x'$ implies $F(x,w) \preceq F(x',w)$
for any fixed $w \in \WW$.  We let $Q$ represent the 
stochastic kernel corresponding to \eqref{eq:srs}, so that
$Q(x,B) = \phi \setntn{w \in \WW}{ F(x,w) \in B}$ for all 
$x \in \XX$ and $B \in \bB$.
Since $F$ is increasing, the kernel $Q$ is increasing.  Hence Theorem~\ref{t:bk}
applies.   We can obtain a lower bound on $\epsilon$ in \eqref{eq:defc} by
calculating
\begin{equation}\label{eq:lbe}
    e \coloneq
    \inf
        \left\{
            \int \int
            \1\{ F(x', w') \leq F(x, w) \} \phi(\diff w) \phi(\diff w')
         \, : \,
         (x, x') \in C \times C
        \right\}.
\end{equation}
Indeed, if $W$ and $W'$ are drawn independently from $\phi$, then 
$X = F(x, W)$ is a draw from $Q(x, \cdot)$ and 
$X' = F(x', W)$ is a draw from $Q(x', \cdot)$.  Hence
\begin{equation}\label{eq:ele}
    e
    = \PP\{X' \preceq X\} 
    \leq \alpha(Q(x, \cdot), Q(x', \cdot)) 
    \leq \epsilon.
\end{equation}

\subsection{Example: TCP Window Size Process}\label{ss:tcp}

To illustrate the method in Section~\ref{ss:srs}, we consider the TCP window
size process (see, e.g., \cite{bardet2013total}), which 
has embedded jump chain 
$X_{t+1} = a (X_t^2 + 2E_{t+1})^{1/2}$.
Here $a \in (0,1)$ and $(E_t)$ is {\sc iid} exponential with unit rate.
If $C = [0, c]$, then drawing $E, E'$ as independent
standard exponentials and using \eqref{eq:ele},
\begin{equation*}
    \epsilon \geq \inf_{0 \leq x, y \leq c}
            \PP\{ a \sqrt{y^2 + 2E'} \leq a \sqrt{x^2 + 2E} \} 
            = \PP\{ \sqrt{c^2 + 2E'} \leq \sqrt{2E} \} .
\end{equation*}
Since $E' - E$ has the Laplace-$(0,1)$ distribution, we get
\begin{equation*}
   1 - \epsilon \leq 
            \PP\{ c^2 + 2E' > 2E \} 
            = \PP\{ E' - E  > c^2/2 \} 
            = \frac{1}{2} \exp(- c^2/2 ). 
\end{equation*}

\subsection{Example: When Minorization Fails}\label{ss:wmf}

We provide an elementary scenario where Theorem~\ref{t:bk} provides a
usable bound while the
minorization based methods described in Section~\ref{ss:tv} do not.
Let $\QQ$ be the rational numbers, let $\XX = \RR$, and assume that
\begin{equation*}
    X_{t+1} = \frac{X_t}{2} + W_{t+1}
    \quad 
    \text{where $W_t$ is {\sc iid} on $\{0,1\}$ and $\PP\{W_t = 0\} = 1/2$}.
\end{equation*}
Let $C$ contain at least one rational and one irrational number.
Let $\mu$ be a measure on the Borel sets of $\RR$ obeying
$\mu(B) \leq Q(x, B) = \PP\{x/2 + W \in B\}$ for all $x \in C$ and Borel sets $B$.
If $x$ is rational, then $x/2 + W \in \QQ$ with probability one, so $\mu(\QQ^c)
\leq Q(x, \QQ^c) = 0$.  Similarly, if $x$ is irrational, then
$x/2 + W \in \QQ^c$ with probability one, so $\mu(\QQ) \leq Q(x, \QQ) = 0$. 
Hence $\mu$ is the zero measure on $\RR$.  Thus, we cannot take
a $\hat \epsilon > 0$ and probability measure $\nu$ obeying the minorization
condition \eqref{eq:mc}.
On the other hand, letting $V(x) = x + 1$ and $d=1$, so that $C = \{V \leq 2\} =
[0, 1]$, the value $e$ from \eqref{eq:lbe} obeys
   $e = \PP\{1/2 + W \leq W'\}
      = \PP\{W' - W \geq 1/2 \} = \frac{1}{4}$.
Hence, by \eqref{eq:ele}, the constant $\epsilon$ in Theorem~\ref{t:bk} is positive.

\subsection{Example: Wealth Dynamics}

Many economic models examine wealth dynamics in the presence of credit market
imperfections (see, e.g., \cite{antunes07start}). These often result in dynamics
of the form
\begin{equation}\label{eq:sigma}
    X_{t+1} = \eta_{t+1} \, G(X_t) + \xi_{t+1},
    \quad (\eta_t) \iidsim \phi,
    \quad (\xi_t) \iidsim \psi.
\end{equation}
Here $(X_t)$ is some measure of household wealth, $G$ is a 
function from $\RR_+$ to itself and $(\eta_t)$ and $(\xi_t)$ are independent
$\RR_+$-valued sequences. The function $G$ is increasing, since greater current
wealth relaxes borrowing constraints and increases financial income.
We assume that there exists a $\lambda < 1$ such that $\EE \,
\eta_t G(x) \leq \lambda x$ for all $x \in \RR_+$,  and, in addition, that
$\bar \xi \coloneq \EE \xi_t < \infty$.

Let $Q$ be the stochastic kernel corresponding to \eqref{eq:sigma}.  With $V(x)
= x + 1$, we have
\begin{equation}\label{eq:fd}
    QV(x) 
    = \EE [ \eta_{t+1} \, G(x) + \xi_{t+1} + 1 ]
    \leq \lambda x + \bar \xi + 1
    \leq \lambda V(x) + \bar \xi + 1.
\end{equation}

Fixing $d \in \RR_+$ and setting $C = \{V \leq d\} = [0, d]$, we can obtain $e$
in \eqref{eq:lbe} via
\begin{equation*}
    e 
    = \PP
    \{\eta' G(d) + \xi' \leq \eta G(0) + \xi\}
    \quad \text{when } \quad (\eta', \xi', \eta, \xi) \sim \phi \times \psi \times \phi \times \psi.
\end{equation*}
This term will be strictly positive under suitable conditions,
such as when $\psi$ has a sufficiently large support.
Combining \eqref{eq:ele} and \eqref{eq:fd} with the bound in Theorem~\ref{t:bk},
we have, for any $\mu$ and $\mu'$ in $p\bB$ and $j, t \in \NN$ with $j \leq t$, 
\begin{equation*}
    \kappa( \mu Q^t, \mu' Q^t)
      \leq 
          (1-e)^j + 
          \frac{1}{2}\left(\lambda + \frac{2 \bar (\xi + 1)}{d}\right)^t 
         d^{j-1} 
        H(\mu, \mu').
\end{equation*}
where $H(\mu, \mu') \coloneq \left(\int x \mu(\diff x) + \int x \mu'(\diff x) \right)/2$.

Notice that, for this model, the lack of smooth mixing and continuity implies
that neither total variation nor Wasserstein distance bounds can be computed
without additional assumptions.

\bibliographystyle{plain}
\bibliography{localbib}

\appendix

\section{Proof of Lemma~\ref{l:ggc}}

Let the conditions of Lemma~\ref{l:ggc} hold and let $Q$, $\hat Q$ and
$((X_t, X'_t))_{t \geq 0}$ be as described above.  We assume in what follows
that $H(\mu, \mu')$ is finite, since otherwise Lemma~\ref{l:ggc} is trivial. 
Setting $W(x, x') \coloneq (V(x)+V(x'))/2$, we have
\begin{equation}\label{eq:wdr}
    \hat Q W (x, x') = 
      \frac{Q V(X_t) + Q V(X'_t)}{2}
      \leq \lambda W (x, x') + \beta.
\end{equation}
Now define
\begin{equation*}
    M_t \coloneq \gamma^{-t} d^{-N_{t-1}} W(X_t, X_t')
    \quad \text{for } t \geq 0 \text{ with }
    N_{-1} \coloneq 1.
\end{equation*}
We claim that $(M_t)_{t \geq 0}$ is an $(\fF_t)$-supermartingale. Clearly $(M_t)$ is adapted.
To see that $\EE[ M_{t+1} \given \fF_t] \leq M_t$
holds $\PP$-almost surely,\footnote{This inequality implies
integrability of $M_t$ because then $\EE M_t \leq \EE M_0$ and $\EE M_0 =
H(\mu, \mu')$ is finite by assumption.} let $F_t \coloneq \1 \{ (X_t, X_t') \in C \times C \}$, and
let $F_t^c \coloneq 1 - F_t$, so that 
\begin{equation}
    \label{eq:bce0}
    \EE[ M_{t+1} \given \fF_t]
      =  \EE [ M_{t+1} F_t \given \fF_t] +  \EE [ M_{t+1} F_t^c \given
      \fF_t].
\end{equation}
On $C \times C$ we have $W \leq d$, so, by \eqref{eq:wdr}, $\EE [W(X_{t+1},X'_{t+1}) \given \fF_t]  \leq \lambda d + b = d \gamma$.
Therefore,
\begin{equation*}
    \EE [ M_{t+1} F_t \given \fF_t]
       = \gamma^{-(t+1)} d^{-N_t} 
       \EE [W(X_{t+1}, X'_{t+1}) \given \fF_t ] F_t
       \leq \gamma^{-t} d^{-N_t + 1} F_t. 
\end{equation*}
Also, on $F_t$ we have $N_t = N_{t-1} + 1$.  Using this fact and $1 \leq W$ yields
\begin{equation}\label{eq:mg1}
    \EE [ M_{t+1} F_t \given \fF_t]
    \leq \gamma^{-t}  d^{-N_{t-1}} F_t 
       \leq M_t F_t .
\end{equation}

Turning to the term $\EE [ M_{t+1} F_t^c \given \fF_t]$,   observe that
on $F_t^c$ we have $W \geq d/2$, so, using \eqref{eq:bqv1} again,
\begin{equation*}
  \frac{\hat Q W}{W}
  \leq
        \lambda  + \frac{\beta}{W} 
  \leq  
      \lambda  + \frac{2\beta}{d} 
  = \gamma .
\end{equation*}
Therefore, $\EE[ W(X_{t+1}, X'_{t+1}) F_t^c \given \fF_t] \leq \gamma
W(X_t,X'_t) F_t^c$. Combining this bound with the fact that 
$N_t = N_{t-1}$ on $F_t^c$ yields
\begin{equation*}
  \EE[ M_{t+1} F_t^c \given \fF_t]
      = \gamma^{-(t+1)} d^{-N_t} \EE [W(X_{t+1},X'_{t+1}) \given \fF_t]F_t^c
      \leq \gamma^{-t} d^{-N_{t-1}} F_t^c
      \leq M_t F_t^c,
\end{equation*}
where the last inequality used $1 \leq W$.
Together with \eqref{eq:mg1}  and \eqref{eq:bce0}, this inequality gives $\EE [
M_{t+1} \given \fF_t] \leq M_t$, so $(M_t)$ is a supermartingale as claimed.

Now fix $t \in \NN$ and $j \leq t$.  Since $d \geq 1$ we have
\begin{equation*}
   \PP\{ N_t < j \}
   \leq \PP\{ N_{t-1} \leq j-1 \}
   = \PP\{ d^{-N_{t-1}} \geq d^{-j+1} \}.
\end{equation*}
From Chebychev's inequality, $1 \leq W$ and the supermartingale property,
the last term is dominated by
\begin{equation*}
   d^{j-1} \EE d^{-N_{t-1}}
   \leq \gamma^t d^{j-1} \EE [ \gamma^{-t} d^{-N_{t-1}} W(X_t, X'_t) ]
   = \gamma^t d^{j-1} \EE [ M_t ] 
   \leq \gamma^t d^{j-1} \EE [ M_0 ] .
\end{equation*}
The last term is just
$\gamma^t d^{j-1} H(\mu, \mu')$,
so the claim in Lemma~\ref{l:ggc} is now proved.

\end{document}